\documentclass[oneside, letter]{cleanShai} 
\author{Burden \& Revzen}
\newcommand{\ShortTitle}{Differentiating a $PC^r$ flow}
\newcommand{\LongTitle}{Computing the Bouligand derivative of a class of piecewise--differentiable flows}
\title{\ShortTitle}
\journalname{\LongTitle}
\usepackage{graphicx}
\graphicspath{{figures/}{./}}
\usepackage[table]{xcolor}
\usepackage{amssymb}
\usepackage{amsmath}
\usepackage{amsthm}
\usepackage{enumitem}

\newtheorem{theorem}{Theorem}[section]

\newtheorem{lemma}[theorem]{Lemma}
 
\usepackage{dsfont} 
\usepackage[mathscr]{eucal}
\usepackage[compress]{natbib}
\bibpunct{[}{]}{;}{a}{,~}{,} 

\newcommand{\NOcitep}[2][]{}
\newcommand{\NOcitet}[2][]{}
\newcommand{\concept}[1]{\textit{#1}}

\newcommand{\R}{\mathbb{R}}
\newcommand{\F}{\mathcal{F}}
\newcommand{\Cont}{\mathcal{C}}

\newcommand{\Z}{\mathbb{Z}}

\newcommand{\N}{\mathbb{N}}
\newcommand{\B}{\mathcal{B}}

\newcommand{\T}{\mathsf{T}}
\newcommand{\vf}{f}
\newcommand{\flow}{\Phi}
\newcommand{\Ek}{\mathcal{E}_k}
\newcommand{\Cr}{\mathcal{C}^r}
\newcommand{\PCr}{\mathcal{P}\mathcal{C}^r}
\newcommand{\ECr}{\mathcal{P}\mathcal{C}^r}
\newcommand{\Deriv}{\mathrm{\mathbf{D}}}

\DeclareMathOperator{\cone}{Cone}

\newcommand{\veps}{\varepsilon}

\newcommand{\ones}{\mathds{1}}

\newcommand{\conv}{\operatorname{conv}}

\newcommand{\sgn}{\operatorname{sign}}

\usepackage{marvosym}
\usepackage[
    hypertexnames,%
    citecolor=blue,%
    colorlinks=true,%
    linkcolor=red%
]{hyperref}
\usepackage[all]{hypcap}

\fancypagestyle{followingpage}{%
\fancyhf{} 
\fancyhead[L]{\Large \thepage}
\fancyfoot[C]{\large {Revzen \& Burden}}
\fancyhead[C]{\textsc{\LongTitle}}

}

\usepackage{placeins}

\makeatother

\providecommand{\doi}[1]{DOI \href{http://dx.doi.org/#1}{#1}}

\begin{document}
\sffamily
\pagestyle{followingpage}
\graphicspath{{figures/}}
\title{\LongTitle}
\begin{abstract}
Event--selected $C^r$ vector fields yield piecewise--differentiable flows~\citep{Burden-2015-multi}, which possess a continuous and piecewise--linear Bouligand (or B--)derivative~\citep[Prop.~4.1.3]{Scholtes2012}; 
here we provide an algorithm for computing this B--derivative.
The number of ``pieces'' of the piecewise--linear B--derivative is factorial ($d!$) in the dimension ($d$) of the space, precluding a polynomial--time algorithm. 
We show how an exponential number ($2^d$) of points can be used to represent the B--derivative as a piecewise--linear homeomorphism in such a way that evaluating the derivative reduces to linear algebra computations involving a matrix constructed from $d$ of these points.
\end{abstract}

\maketitle
\tableofcontents

\section{Background}

This brief note serves as an addendum to~\citep{Burden-2015-multi} that provides a compact representation of a nonclassical derivative operator, namely, the \emph{Bouligand} (or \emph{B--})derivative of the piecewise--differentiable flow associated with a class of vector fields with discontinuous right--hand--sides, termed (in~\citep{Burden-2015-multi}) \emph{event--selected $\Cr$}.
This B--derivative is a piecewise--linear map.
In~\citep[Sec.~7]{Burden-2015-multi} we provided the means to evaluate the B--derivative in a chosen direction (i.e. on a given tangent vector). 
In this note we construct a triangulation for the B--derivative, thus representing it as a piecewise--linear homeomorphism using the techniques from~\citep[Sec.~2]{Groff2003}.

Before proceeding, we informally recapitulate definitions and results from~\citep{Burden-2015-multi} to provide notational and conceptual context for what follows. 
Given an open subset $D\subset\R^d$ and order of differentiability $r\in\N$, the vector field $\vf:D\to\T D$ is termed \emph{event--selected $\Cr$ at $\rho\in D$} if there exists an open set $U\subset D$ containing $\rho$ and a collection of \emph{event functions} $\{h_k\}_{k=1}^n\subset \Cr(U,\R)$ for $\vf$~\citep[Def.~1]{Burden-2015-multi} such that  for all $b\in\B = \{-1,+1\}^n$, with
\[ D_b = \{x\in\R^d \mid \forall k\in\{1,\dots,n\} : b_k (h_k(x) - h_k(\rho)) \ge 0 \}, \]
$\vf|_{\operatorname{Int} D_b}$ admits a $\Cr$ extension $\vf_b:U\to TU$~\citep[Def.~2]{Burden-2015-multi}.
Note that $\vf$ is allowed to be discontinuous along each \emph{event surface} $\Ek = h_k^{-1}(h_k(\rho))$, but is $\Cr$ elsewhere.

If $\vf$ is event--selected $\Cr$ at $\rho\in D$, then~\citep[Thm.~4]{Burden-2015-multi} ensures there exists a piecewise--$\Cr$ flow $\flow:\F\to D$ for $\vf$ defined on an open subset $\F\subset\R\times D$ containing $(0,\rho)$.
The notion of piecewise--differentiablity used in~\citep{Burden-2015-multi} and in what follows is due to Robinson~\citep{Robinson1987}; for a highly--readable exposition of piecewise--$\Cr$ functions and their properties, we refer the interested reader to~\citep{Scholtes2012}.
In short, a continuous function is piecewise--$\Cr$ (or $\PCr$) if its graph is everywhere locally covered by the graphs of a finite number of $\Cr$ functions (termed \emph{selection functions}).
Piecewise--$\Cr$ functions always possess a continuous first--order approximation termed a \emph{Bouligand} (or \emph{B--})derivative~\citep[Prop.~4.1.3]{Scholtes2012}; due to the local finiteness of selection functions, the B--derivative of $\PCr$ functions is piecewise--linear.
The remainder of this note will be devoted to constructing a triangulation for the piecewise--linear B--derivative of the $\PCr$ flow $\flow$.

\section{Normal forms}
Let $f:\,U \to \T U$ 
be an event--selected $\Cont^r$ vector field at $\rho\in\R^n$ with respect to $h:\,U\to\R^d$ where $U$ is a neighborhood of $\rho$.
By assumption, $f$ is $\Cont^r$ everywhere except (perhaps) $h^{-1}(h(\rho)) = \{ x \in U ~|~ \exists k:\, h_k(x) =h_k(\rho)\}$, the ``transition surfaces''.
Also by assumption, $h_k(\rho)$ is a regular value 
for event function $h_k$ and $\Deriv h_k \cdot f > \veps$ everywhere in $U$ for some $\veps > 0$.
If $n \neq d$, the system can be embedded via the technique in~\citep[Remark~4]{Burden-2015-multi} in a higher--dimensional system where the dimension of the state space equals the number of event surfaces, and the matrix $Dh(\rho)$ is invertible.

\subsection{Piecewise--constant sampled systems}
Let $\vf:\,D\to\T U$ be $\ECr$ with respect to $h:\,U\to\R^n$ at $\rho$. 
From the previous section, we assume without loss of generality that $U$ is $n$-dimensional. 
We refer to the zero level sets of the components of $h$, $\Ek := \{ x \in U \,|\,h_k(x) = 0 \}$ as \concept{local sections}~\cite[Def.~1]{Burden-2015-multi}.
Since $\vf$ is $\ECr$, zero is a regular value of each of the $h_k(\cdot)$ functions, and thus $\Ek$ are embedded codimension--1 submanifolds. 

Let $b\in \B = \{-1,+1\}^n$ be a corner of the \emph{hypercube} $\{-1,+1\}^n$.
Define $F_b := \lim_{\alpha\to 0^+} f(h^{-1}(\alpha b))$ the \concept{corner value} of $f$ at the corner $b$.
Note that by construction, all coordinates of $F_b$ are positive and larger than $\veps$.
Extend (by slight abuse of notation) to $F(x) := F_\rho\left(\sgn \Deriv h(\rho)\cdot(x-\rho)\right)$.
The flow ${\hat \Phi}^t(\cdot)$ of $F(\cdot)$ near $\rho$, has transition manifolds which are affine subspaces of co-dimension 1, tangent to the transition manifolds of the original system at $\rho$.
Furthermore,~\cite[Eqn.~(63)]{Burden-2015-multi} shows that the sampled system's flow provides a first--order approximation for that of the original system at $\rho$.

\section{The structure of corner limit flows}

\newcommand{\Imp}{\tau}
We recall from~\citep[Thm.~7]{Burden-2015-multi} that the time--to--impact any local section of an event--selected $C^r$ vector field is a piecewise--$C^r$ function.
For each surface index $k\in\{1,\dots,n\}$, let $\Imp_k:V_k\to\R$ denote the time--to--impact map for the event surface $\Ek = h_k^{-1}(\R)$ defined over a neighborhood $V_k$ containing $\rho$.
Letting $\Imp = (\Imp_1,\dots,\Imp_n):V \to \R^n$ denote the composite function defined over $V = \cap_{k=1}^n V_k$%
\footnote{We note that $V$ is open since each $V_k$ is open and nonempty since $\rho\in V$.},
it follows that for any point $x\in V$ and index $k\in\{1,\ldots,n\}$, $e_k^\T \Imp(x) = \Imp_k(x)$ is the time required for $x$ to flow to surface $\Ek$, i.e. $h_k(\flow(\Imp_k(x),x))=0$ for the flow $\flow:\F\to D$ of the vector field $f$.
Furthermore, it is clear that $\tau$ is injective and its image is an open set, whence Brouwer's Open Mapping Theorem~\citep{Brouwer1911, Hatcher2002} implies $\tau$ is a homeomorphism onto its image.

\begin{lemma}\label{lem:diag}
Time--to--impact maps ($\Imp$ in the preceding paragraph) have the following properties:
\begin{enumerate}
\item they are $\Cr$ on any submanifold that encounters events in the same order;
\item they take the zero level sets of the event functions ($h_k$) to the standard arrangement;
\item \label{itm:allones} they take the vector field ($f$) to the constant vector field $\ones$ (referred to as the \concept{diagonal flow} hereon);
\item \label{itm:impacts} for the diagonal flow, the time--to--impact is $x \mapsto -x$;
\item \label{itm:linear} for sampled systems, the time--to--impact map is piecewise--linear.
\end{enumerate}
\end{lemma}
\begin{proof}
Properties 1--4 are direct; property 5 follows from~\citep[Remark~4]{Burden-2015-multi}
\end{proof}

From lemma \ref{lem:diag} we conclude that the flow $\Phi^t(\cdot)$ in a neighborhood of $\rho$ is $\PCr$ conjugate through $\Imp(\cdot)$ to the diagonal flow.
By the definition of $\Imp$, an event $h_k(\cdot)$ occurs in the original coordinates at $x$ if, and only if, $e_k^\T \Imp(x) = 0$.

It remains to analyze the structure of the diagonal flow.
\subsection{The standard cone}

The \concept{cone span} of a set of vectors $X \subseteq \R^n$ is given by $\cone X := \{ y \in \R^n \,|\, \sum_{i=1}^m \alpha_i x_i,\, \alpha_i\in \R^+\}$.
In this section we show that the diagonal flow in $\R^n$ comprises $n\!$ identical cones, whose form we make computationally explicit.
The interior of each of these cones consists of all the points whose transitions happen in a specific order, as specified by a permutation $\sigma \in S^n$.

Assume hereon that $\dot y = \ones$ is the diagonal flow, conjugate through $\Imp(\cdot)$ to $\dot x = f(x)$, i.e. $\Deriv \Imp(x) \cdot f(x) = \ones$ everywhere.
For any $y \in \Imp(U)$, let $\sigma_y \in S^n$ be a permutation that sorts the elements of $y$.
This permutation can be represented by a permutation matrix $Z_\sigma$ such that $z := Z_\sigma y$ satisfies:
\begin{align}
\forall 0<i<j\leq \dim U:\, (i<j) \to \left( e_i^\T z \leq e_j^\T z \right)
\end{align}
\newcommand{\Grf}{\mathcal{G}}
We define the group of permutation matrices $\Grf := \{ Z_\sigma \,|\, \sigma \in S^n \}$; this is a representation of $S^n$ which is valid over all fields.

\newcommand{\Cn}{\mathcal{K}}
We define the sets
\begin{align}\label{eqn:zmon}
\Cn_\sigma := \left\{ Z^{-1}_\sigma z \,\middle|\, z\in \R^n,  z_1 \leq z_2 \leq \ldots \leq z_{n-1} \leq z_n \right\}
\end{align}
and denote by $\Cn$ the set associated with the identity permutation.

By construction, $\Cn_\sigma$ are exactly the points whose impact times are (weakly) in the order specified by $\sigma$.
If impact times are different, i.e. strongly in the order specified by $\sigma$, the inequalities in \ref{eqn:zmon} are strong, and the corresponding point is an interior point of $\Cn_\sigma$.

\subsection{Constructing $\Cn$ as cone span}

\newcommand{\Lt}{M^U}
\newcommand{\Lti}{\Delta^U}

In this section we will use addition and scalar multiplication in a setwise sense, i.e. $AB$ is the set of all products of elements of $A$ and of $B$, $A+B$ is the Minkowski sum -- the set of all possible sums comprising an element of $A$ and an element of $B$.

We define the vectors that form the columns of a lower triangular matrix $\Lt$ with elements zero to be $s_k := \frac{n}{n+1-k}\sum_{i=k}^n e_k$, as follows:
\begin{align}
\Lt := \left[\begin{array}{ccccc}
 1  	&\cdots & \cdots 	& 1  		&  1   \\
 0		& n/(n-1)&n/(n-1)& \ddots 		& n/(n-1)\\
 \vdots &\ddots & \ddots 	& \ddots 	& \vdots\\
 \vdots	&		& \ddots	& n/2 		& n/2 	\\
 0 		&\cdots	& \cdots 	& 0 		& n		
\end{array}\right]
&&
\Lti := \frac{1}{n}\left[\begin{array}{ccccc}
 n 		& -(n-1) 	&0		& \cdots 	& 0 	\\
 0 		& (n-1)		&\ddots	& \ddots   	& \vdots\\
 \vdots & \ddots 	&\ddots	& -2		& 0 	\\
 \vdots & 			&\ddots	& 2 		& -1	\\
 0 		&\cdots 	&\cdots & 0 		& 1  
\end{array}\right]
\end{align}

Note that $\ones \cdot s_k = n$ for all $k$, $s_1 = \ones$.

$\Cn$ is a cone -- all positive multiples of $z$ will satisfy the same inequality as $z$ in \ref{eqn:zmon}, i.e. $\R^+\Cn = \Cn$.
$\Cn$ is also invariant under the diagonal flow, i.e. $\Cn=\R \ones + \Cn$.
Using the $s_k$ vectors we note that $\Cn = \R s_n + \sum_{k=1}^{n-1} \R^+ s_k$.

\newcommand{\A}{\mathcal{A}}
\newcommand{\Qa}{\mathcal{D}}
Define $\Qa_b$ for $b \in \R^n$ as $\Qa_b = \{ x \in \R^b \,|\, \sgn x = \sgn b \}$.
For the diagonal flow, the points of a set $\Qa_b$ are all points that have the same events as for $b$ happen to them in the past (in whatever order), the same events will happen to them as for $b$ (in whatever order) in the future, and the same set of events are currently happening to them as are happening to $b$. 

Let us define the \concept{anti-diagonal} as $\A_0 := \{ p \in R^n \,|\, \ones \dot p = 0 \}$, and similarly define two affine subspaces parallel to $\A_0$ by $\A_\pm := \A_0 \pm \ones$.

\begin{lemma}
The following set-wise equality holds 
\begin{align}
\A_+ \cap \Qa_{\ones} \cap \Cn
&= \conv \{s_1, \ldots, s_n\}\label{eqn:splx}
\end{align}
and describes a simplex of dimension $n-1$.
\end{lemma}
\begin{proof}
We use some facts from convexity: (1) for a convex set $A$, when $B \subseteq A$ then $(\conv B) \subseteq A$; (2) for convex $A,B$ the intersection $A \cap B$ is also convex;
(3) convex closure preserves linear constraints: if points $B$ meet a constraint $\forall b \in B:\, w \cdot b = c$, then so does their convex closure $\forall x \in \conv b:\, w \cdot x = c$. 
Note that the sets $\A_+$, $\Cn$ and $\Qa_{\ones}$ are convex.

First we show the inclusion of the RHS in the LHS of \ref{eqn:splx}.
For all corners $s_k$, $s_k \in \Qa_{\ones}$ because they have non-negative coordinates; they satisfy $\ones \cdot s_k = 1$; and $s_k \in \Cn$ because the coordinates of $s_k$ are non-decreasing.
From the convexity properties mentioned, the inclusion of the RHS in the LHS follows.

It remains to demonstrate the inclusion of the LHS of \ref{eqn:splx} in the RHS. 
Let $z$ be a point in the intersection in the LHS. 
From $z \in \A_+ \cap \Qa_{\ones}$ it follows that $z \cdot \ones = n$, and $\forall k:\, z_k \geq 0$.
Given $z = z \Lti \Lt$ we have that for $a := z \Lti$ we need to show $\sum_{k=1}^n a_k$ is one, and $a_k \geq 0$.

By examining the columns of $\Lti$ we can see that $a_1 = z_1$, and for $k>1$ we have $a_k = \frac{n+1-k}{n}\left(z_k-z_{k-1}\right)$.  
From $z \in \Cn$, we know $z_k\geq z_{k-1}$, giving us that $a_k \geq 0$.
Note that $\ones$ is (by direct examination) a right eigenvector of $\Lti$ with eigenvalue $n^{-1}$, giving
\begin{align}
\sum_{k=1}^n a_k = z \Lti \cdot \ones = \frac{1}{n} z \cdot \ones = 1
\end{align}
proving the inclusion of the LHS in the RHS, and thus the desired equality.
\end{proof}

This allows us to state a more general theorem
\begin{theorem}
\begin{align} \A_+ \cap \Qa_{\ones} \cap \Cn_\sigma
&= \conv \{Z_\sigma s_1, \ldots, Z_\sigma s_n\}\label{thm:partn}
\end{align}
\end{theorem}
\begin{proof}
Convex closure commutes with all linear maps, in particular the $Z_\sigma$ maps that give $\Cn_\sigma = \Z_\sigma \Cn$.
We obtain $\Cn_\sigma = \conv \left\{ Z_\sigma s_k \right\}_{k=1}^n$.
Furthermore $Z_\sigma \A_+ = \A_+$ and $Z_\sigma \Qa_{\ones} = \Qa_{\ones}$, allowing us to conclude the desired result.
\end{proof}

By construction $\left\{\Cn_\sigma \right\}_{ \sigma \in \mathsf{S}_n }$ are a cover of $\R^n$.
We conclude that by knowing how the corners in \ref{thm:partn} map through the flow we may deduce the values at all other points by barycentric (convex) interpolation.

\subsubsection{The sampling points $\{Z_\sigma s_k\}_{k=1}^n$}
\newcommand{\Pts}{\mathcal{P}}
The set of simplex corner points $\Pts := \{Z_\sigma s_k\}_{k=1}^n$ is smaller than may appear at first.
This is due to the fact that the cardinality $\#\{Z_\sigma s_k\}_{\sigma\in\mathsf{S}_n}={n \choose k}$ -- there are only as many of them as there are ways of choosing $k$ of the $n$ coordinates to be zero. 
Overall we obtain $\#\Pts = 2^n$, the number of ways to select which entries will be zero out of $n$ possibilities.
Thus by flowing forward all $2^n$ points of $\Pts$ from $\A_-$ until they impact $\A_+$, we obtain all the necessary information for computing all $n!$ affine transformations that comprise the first order approximation of the flow near the origin.

\subsection{Bypassing the need for the map $\Imp(\cdot)$}

The trajectories going through points of $\Pts$ other than $\ones$ have a unique property: each of them has only two times at which events will occur. For $Z_\sigma s_k$, a set of events occurs simultaneously at time $0$, and all other events occurs simultaneously at time $\frac{-n}{n+1-k}$.

We will now show how to compute $\Imp^{-1}(\Pts)$ for corner limit flows.
Let $p\in\Pts$ such that $p = Z_\sigma s_k$, and $p_i \neq 0$ if and only if $i \in S \subseteq \{1,\ldots,n\}$ (i.e. $S$ is the support of $p$).
Let $b \in \B$ such that $b_i = -1$ for $i\in S$ and $b_i = 1$ otherwise (for $i \in {\bar S}$).
In the diagonal flow, $\Qa_b$ is the set of points for which the events of $S$ have already happened, and the remaining events ${\bar S}$ are yet to occur.

Let us construct a trajectory which at time 0 goes through $q(0) := \Imp^{-1}{p}$.
The state $q(0)$ is on the event surfaces of $h_i$ for $i \in {\bar S}$, i.e. $\forall i \in {\bar S}:\, \nabla h_i (\rho) \cdot q(0) = 0$.
For any positive time $t_+>0$, we have $q(t_+)=q(0) + F(\ones) t_+$.

The remaining events for this trajectory occurred at time $-t := \frac{-n}{n+1-k}$, and in the time interval $(-t,0)$ the trajectory was moving under the influence of $F(b)$.
Thus $q(-t)=q(0)-t F(b)$ is on the remaining event surfaces, i.e. $\forall i \in S:\, \nabla h_i (\rho) \cdot q(0) = t \nabla h_i(\rho) \cdot F(b)$.
For time s$t_-<-t$ the trajectory moved under the influence of $F(-\ones)$, i.e. $q(t_-) = q(0) - t F(b) + (t_-+t) F(-\ones)$, completing the description of the trajectory.

Note that the vector whose $S$ coordinates are of value $t$, and whose $\bar S$ coordinates are $0$ is  
\begin{align}
s(b) := \frac{(b + \ones) n}{2n+2-\ones\cdot(b+\ones)} = Z_\sigma s_k \label{eqn:bones}
\end{align}
Thus we obtain a simplified equation
\begin{align}
\Deriv h (\rho) q(0)  = \mathrm{diag}(s(b)) \, \Deriv h (\rho)\cdot F(b) \label{eqn:nlcorner}
\end{align}
which allows $q(0)$ to be solved for in the original (fully non-linear) coordinates for all $b \in \B$, provided the corner limits $F_b$ and event Jacobian $\Deriv h$ can be computed.

Note also that $\Deriv h$ is never inverted, allowing this computation to be used with event surfaces that are not transverse. 

\subsection{Computing the derivative of the flow}
\newcommand{\Vrt}{\mathcal{V}}
\newcommand{\Xs}{\mathcal{X}}

\begin{figure}
\centering
\fbox{
\includegraphics[width=.85\columnwidth,trim={1cm 9cm 2cm 1cm},clip]{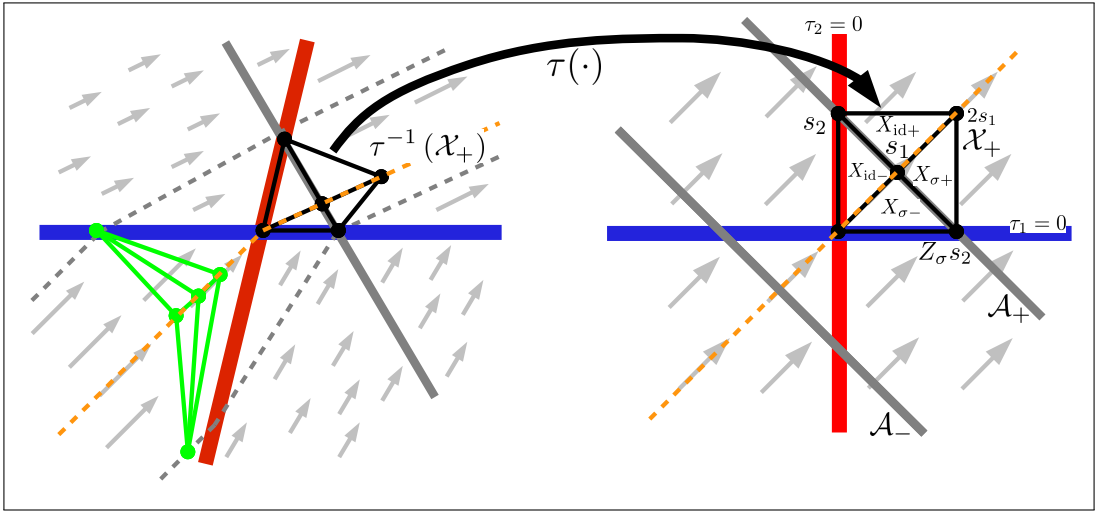}
}
\caption{A corner limit flow [left] is mapped by $\Imp(\cdot)$ into impact time coordinates [right]. %
In impact times coordinates, we define triangulation $\Xs_+$ which comprises simplices that go through known transition sequences ($\text{id}$ and $\sigma$ in this 2D case; black lines).
The differential we seek is given by the piecewise linear homeomorphism produced by carrying this triangulation through the flow [left] from the initial locations $q_k(-T)$ (green) to the final locations $q_k(0)$.
}
\label{fig:cplxs}
\end{figure}
Denote the $q(\cdot)$ of \ref{eqn:nlcorner} and the previous section by $q_b(\cdot)$ to highlight its dependence on $b$.
In addition, define $q_0(0)=\rho$ and for $t_+>0$ take $q_0(t_+) = F(\ones) t_+$ and $q_0(-t_+) = - F(-\ones) t_+$ -- the trajectory of the origin.

We will now define a set of simplexes surrounding the point $q_0(1)$.
The vertices of these simplexes will be 
\begin{align}
\Vrt := \{ q_0(0), q_0(1), q_0(2) \} 
	\cup \left\{q_b(0)\right\}_{b\in\B}
\end{align}

For every permutation $\sigma \in \mathsf{S}_n$ we will have a $n-1$ dimensional face 
\begin{align}
Y_\sigma := \conv \left( 
	\{q_0(1)\} \cup \left\{ q_b(0) \,\middle|\,
	b_{\sigma(1)},\ldots,b_{\sigma(k)} = -1,
    b_{\sigma(k+1)},\ldots,b_{\sigma(n)} = 1,
    \text{for~} 0<k<n\right\}
    \right)\label{eqn:faces}
\end{align}
These faces cover a neighborhood of $q_0(1)$ in the affine subspace $\Imp^{-1}(\A_+)$.

For each face $Y_\sigma$ we will define two simplices $X_{-\sigma}$ and $X_{+\sigma}$ by $X_{-\sigma} = \conv \left( \{q_0(0)\} \cup Y_\sigma\right) $ and $X_{+\sigma} = \conv \left( \{q_0(2)\} \cup Y_\sigma\right)$.
Let $\Xs_+ := \left\{X_{-\sigma},X_{+\sigma}\right\}_{\sigma \in \mathsf{S}_n}$ be the set of all these simplices.
By construction: (1) $\Xs_+$ covers a neighborhood of $q_0(1)$; (2) All points of each simplex $X_{\pm\sigma} \in \Xs_+$ experience events in the same order -- the order designated by $\sigma$. 

At any time $-T < -n$ no events have occurred for any of trajectories $q_b(\cdot)$, i.e. $\forall b\in \B:\, h(q_b(-T))<0$ and furthermore $h(q_0(2-T))<0$.
We can conclude this because the event times are known for each $q_b(\cdot)$ and take on the values $0$ and $\frac{-n}{n+1-k}$ for $1\leq k\leq n$.

Let $\Xs_- := \left\{ \conv \{ q_k(-T) \} \,|\, \conv \{ q_k(0) \} \in \Xs_+ \right\}$, i.e. $\Xs_-$ are the simplices created by taking the convex closure of the corners of simplices of $\Xs_+$ after carrying them back to time $-T$.
All that remains is to construct the piecewise linear homeomorphism mapping $\Xs_-$ to $\Xs_+$ using the techniques of \citep[Sec.~2]{Groff2003}.
In the parlance of that reference, our $q_k(0)$ are the set of $q$ points; our $q_k(-T)$ points are the set of $p$ points; the simplicial complex structure $\Sigma$ is given by equation \ref{eqn:faces}.
Figure \ref{fig:cplxs} shows a example 2-dimensional case.

\section{Acknoledgement}
This work was supported by ARO W911NF–14–1–0573

\begin{thebibliography}{6}
\providecommand{\natexlab}[1]{#1}
\providecommand{\url}[1]{\texttt{#1}}
\expandafter\ifx\csname urlstyle\endcsname\relax
  \providecommand{\doi}[1]{doi: #1}\else
  \providecommand{\doi}{doi: \begingroup \urlstyle{rm}\Url}\fi

\bibitem[Brouwer(1911)]{Brouwer1911}
L~E~J Brouwer.
\newblock Beweis der invarianz desn-dimensionalen gebiets.
\newblock \emph{Mathematische Annalen}, 71\penalty0 (3):\penalty0 305--313,
  1911.
\newblock ISSN 0025-5831.
\newblock \doi{10.1007/BF01456846}.
\newblock URL \url{http://dx.doi.org/10.1007/BF01456846}.

\bibitem[Burden et~al.(2016)Burden, Sastry, Koditschek, and
  Revzen]{Burden-2015-multi}
S~A Burden, S~S Sastry, D~E Koditschek, and S~Revzen.
\newblock Event-selected vector field discontinuities yield
  piecewise-differentiable flows.
\newblock \emph{SIAM Journal of Applied Dynamical Systems}, 15\penalty0
  (2):\penalty0 1227--1267, 2016.
\newblock \doi{10.1137/15M1016588}.

\bibitem[Groff et~al.(2003)Groff, Khargonekar, and Koditschek]{Groff2003}
R~E Groff, P~P Khargonekar, and D~E Koditschek.
\newblock A local convergence proof for the minvar algorithm for computing
  continuous piecewise linear approximations.
\newblock \emph{SIAM journal on numerical analysis}, 41\penalty0 (3):\penalty0
  983--1007, 2003.
\newblock ISSN 0036-1429.
\newblock \doi{10.1137/S0036142902402213}.

\bibitem[Hatcher(2002)]{Hatcher2002}
A~Hatcher.
\newblock \emph{Algebraic topology}.
\newblock Cambridge University Press, 2002.

\bibitem[Robinson(1987)]{Robinson1987}
S~M Robinson.
\newblock Local structure of feasible sets in nonlinear programming, part
  {III}: Stability and sensitivity.
\newblock In \emph{Nonlinear Analysis and Optimization}, volume~30 of
  \emph{Mathematical Programming Studies}, pages 45--66. Springer Berlin
  Heidelberg, 1987.
\newblock ISBN 9783642009303.
\newblock \doi{10.1007/BFb0121154}.
\newblock URL \url{http://dx.doi.org/10.1007/BFb0121154}.

\bibitem[Scholtes(2012)]{Scholtes2012}
S~Scholtes.
\newblock \emph{Introduction to Piecewise Differentiable Equations}.
\newblock SpringerBriefs in Optimization. Springer New York, 2012.
\newblock ISBN 978-1-4614-4340-7.
\newblock \doi{10.1007/978-1-4614-4340-7}.

\end{thebibliography}

\end{document}